\documentclass{svmult}
\usepackage{amscd,amssymb,mathabx}
\usepackage[all]{xy}

\newcommand{\C}{\mathbb{C}}

\newcommand{\RR}{{\mathcal R}}

\newcommand{\F}{{\mathcal{F}}}

\newcommand{\g}{{\mathfrak{g}}}
\newcommand{\h}{{\mathfrak{h}}}
\newcommand{\gl}{{\mathfrak{gl}}}
\renewcommand{\sl}{{\mathfrak{sl}}}

\newcommand{\solv}{{\mathfrak{sol}}}

\newcommand{\ad}{{\rm{ad}}}
\newcommand{\id}{{\rm{id}}}
\newcommand{\im}{{\rm{im}}}
\newcommand{\Hom}{{\rm{Hom}}}
\newcommand{\Lie}{{\rm{Lie}}}
\newcommand{\rank}{{\rm{rank}}}

\newcommand{\abs}[1]{\left| #1 \right|}

\def\dot{\mathchar"013A}  
\newcommand{\hdot}{{\raise1pt\hbox to0.35em{\huge $\dot$}}}

\newcommand{\cdga}{\ensuremath{\mathsf{cdga}}}
\newcommand{\dgla}{\ensuremath{\mathsf{dgla}}}

\begin{document}

\title*{Non-abelian resonance: product and coproduct formulas}

\author{Stefan Papadima \and Alexander~I.~Suciu}

\institute{%
Stefan Papadima \at 
Simion Stoilow Institute of Mathematics, 
P.O. Box 1-764, RO-014700 Bucharest, Romania \\
\email{Stefan.Papadima@imar.ro} \\
Partially supported by the Romanian Ministry of 
National Education, CNCS-UEFISCDI, grant PNII-ID-PCE-2012-4-0156
\and 
Alexander~I.~Suciu \at 
Department of Mathematics,
Northeastern University, Boston, MA 02115, USA \\
\email{a.suciu@neu.edu} \\
Partially supported by 
NSF grant DMS--1010298 and NSA grant H98230-13-1-0225
}
\maketitle

\abstract{We investigate the resonance varieties attached to a 
commutative differential graded algebra and to a representation 
of a Lie algebra, with emphasis on how these varieties behave 
under finite products and coproducts.}

\section{Introduction}
\label{sect:intro}

Resonance varieties emerged as a distinctive object of study in the 
late 1990s, from the theory of hyperplane arrangements.  Their  
usefulness became apparent in the past decade, when 
a slew of applications in geometry, topology, group theory, 
and combinatorics appeared.  

The idea consists of turning the cohomology ring of 
a space $X$ into a family of cochain complexes, 
parametrized by the first cohomology group $H^1(X,\C)$, 
and extracting certain varieties $\RR^i_m(X,\C)$ 
from these data, as the loci where the cohomology of those 
cochain complexes jumps.  Part of the importance of 
these  resonance varieties is their close connection with 
a different kind of jumping loci: the characteristic 
varieties of $X$, which record the jumps in homology with 
coefficients in rank $1$ local systems. 

In recent years, various generalizations of these notions 
have been introduced in the literature, for instance in  
\cite{DPS, DP, PS-jump, MPPS}.   The basic idea now 
is to replace the cohomology ring of a space 
by an algebraic analogue, to wit, a commutative, differential 
graded algebra $(A,d)$, and to replace the coefficient group $\C$ by 
a finite-dimensional vector space $V$, endowed with a representation 
$\theta\colon \g\to \gl(V)$, for some finite-dimensional Lie algebra $\g$. 
In this setting, the parameter space for the higher-rank resonance 
varieties, $\RR^i_m(A, \theta)$, is no longer $H^1(A)$, but rather, 
the space of flat, $\g$-valued connections on $A$, which, according 
to the results of Goldman and Millson from \cite{GM}, is the natural 
replacement for the variety of rank $1$ local systems on $X$.  

In a previous paper, \cite[\S 13]{PS-plms}, we established some basic 
product and coproduct formulas for the classical resonance varieties. 
In this note, we extend those results to the non-abelian case, 
using some of the machinery developed in \cite{MPPS}. 
In Theorem \ref{thm:res prod}, we give a general upper bound 
on the varieties $\RR^i_1(A\otimes \bar{A},\theta)$ in terms of the 
resonance varieties of the factors and the space of $\g$-flat 
connections on the tensor product of the two \cdga's. In 
Theorem \ref{thm:prodres2}, we improve this bound to an
equality of a similar flavor, in the case when the respective 
\cdga's have zero differentials, and $\g$ is either $\sl_2$ or $\solv_2$.  
Finally, in Corollary \ref{cor:coprodg1} and Theorem \ref{thm:res coprod} 
we give precise formulas for the varieties $\RR^i_1(A\vee \bar{A},\theta)$ 
associated to the wedge sum of two \cdga's.  

\section{Flat connections and holonomy Lie algebras}
\label{sect:flat holo}

We start by introducing some basic notions (\cdga's, flat connections, 
holonomy Lie algebras), following in rough outline the exposition 
from \cite{MPPS}. 

\subsection{Differential graded algebras and Lie algebras}
\label{subsec:cdga}

Let $A=(A^{\hdot},d)$ be a commutative, differential 
graded algebra (\cdga) over the field of 
complex numbers, that is, a positively-graded $\C$-vector 
space $A=\bigoplus_{i\ge 0} A^i$, endowed with a 
graded-commutative multiplication 
map $\cdot\colon A^i \otimes A^j \to A^{i+j}$ and a differential 
$d\colon A^i\to A^{i+1}$ satisfying $d(a\cdot b) = da\cdot b 
+(-1)^{i} a \cdot db$, for every $a\in A^i$ and $b\in A^j$.

We will assume throughout that $A$ is connected, i.e., $A^0=\C$,  
and of finite $q$-type, for some $q\ge 1$, i.e., $A^i$ is 
finite-dimensional, for all $i\le q$. Let $Z^i(A)=\ker (d\colon A^i\to A^{i+1})$, 
$B^i(A)=\im (d\colon A^{i-1}\to A^{i})$, and 
$H^i(A)=Z^i(A)/B^i(A)$. 
For each $i\le q$, the dimension of this vector space, 
$b_i(A)=\dim_{\C} H^i(A)$, is finite.

Now let $\g$ be a Lie algebra over $\C$. 
On the vector space $A\otimes \g$, we may define a bracket by
$[a\otimes x, b\otimes y]= ab\otimes [x,y]$ 
and a differential given by $\partial (a\otimes x) = d a \otimes x$, 
for $a,b\in A$ and $x,y\in \g$. 
This construction produces a differential graded Lie algebra (\dgla), 
$A\otimes \g = (A^{\hdot}\otimes \g, \partial)$.  It is readily verified 
that the assignment $(A,\g)\leadsto A\otimes \g$ is functorial 
in both arguments. 

\subsection{Flat, $\g$-valued connections}
\label{subsec:flat}

\begin{definition}
\label{def:flat conn}
An element $\omega\in A^1\otimes \g$ is called an {\em infinitesimal, 
$\g$-valued flat connection}\/ on $(A,d)$ if $\omega$ satisfies the 
Maurer--Cartan equation,
\begin{equation}
\label{eq:flat}
\partial \omega +  [\omega,\omega]/2 = 0 . 
\end{equation}
\end{definition}

We will denote by $\F(A,\g)$ the subset of $A^1\otimes \g$ 
consisting of all flat connections. A typical element in $A^1 \otimes \g$ 
is of the form $\omega =\sum_j \eta_j \otimes x_j$, with 
$\eta_j\in A^1$ and $x_j\in \g$; the flatness condition 
amounts to 
\begin{equation}
\label{eq:flat coords}
\sum_{j} d\eta_j \otimes x_j + 
\sum_{j<k} \eta_j \eta_k \otimes [x_j,x_k] =0.
\end{equation}

In the rank one case, i.e., the case when $\g=\C$, the space 
$\F(A,\C)$ may be identified with the vector space 
$H^1(A)=\{\omega \in A^1\mid d\omega=0\}$. In particular, 
if $d=0$, then $\F(A,\C)=A^1$.

The bilinear map $P\colon A^1\times \g\to A^1\otimes \g$, 
$(\eta,g)\mapsto \eta\otimes g$ induces a map 
$P\colon H^1(A)\times \g \to \F(A,\g)$.  
The {\em essentially rank one}\/ part of the set of flat $\g$-connections  
on $A$ is the image of this map:
\begin{equation}
\label{eq:rk1}
\F^1(A,\g)=P(H^1(A)\times \g).
\end{equation}

\subsection{Holonomy Lie algebra}
\label{subsec:holo}

An alternate view of the parameter space of flat connections 
is as follows.  Let $A_i=\Hom (A^i, \C)$ be the dual 
vector space.  Let $\nabla \colon A_2 \to A_1\wedge A_1$ 
be the dual to the multiplication map 
$A^1\wedge A^1\to A^2$, and let $d_1\colon A_2\to A_1$ be 
the dual of the differential $d^1\colon A^1\to A^2$. 

\begin{definition}[\cite{MPPS}]
\label{def:holo cdga}
The {\em holonomy Lie algebra}\/ of a $\cdga$ $A=(A^{\hdot},d)$ is the 
quotient of the free Lie algebra on the $\C$-vector space 
$A_1$ by the ideal generated by the image of 
$\partial_A=d_1 + \nabla$:
\begin{equation}
\label{eq:holo}
\h(A) = \Lie(A_1) / (\im (\partial_A)). 
\end{equation}
\end{definition}

\begin{remark}
\label{rem:holo chen}
In the case when $d=0$, the above definition coincides with  
the classical holonomy Lie algebra $\h(A)=\Lie(A_1)/(\im(\nabla))$ 
of K.T.~Chen \cite{Ch}.  In this situation, $\h(A)$ inherits a natural 
grading from the free Lie algebra, compatible with the Lie bracket. 
Consequently, $\h(A)$ is a finitely-presented, graded Lie algebra, 
with generators in degree $1$ and relations in degree $2$.
\end{remark}

In general, though, the ideal generated by $\im(\partial_A)$ 
is not homogeneous, and the Lie algebra $\h(A)$ is not graded. 
Here is a concrete example, extracted from \cite{MPPS}.

\begin{example}
\label{ex:solv}
Let $A$ be the exterior algebra on generators $x,y$ in 
degree $1$, endowed with the differential given by $d{x}=0$ 
and $d{y}=y\wedge x$, and let $\solv_2$ be the  
Borel subalgebra of $\sl_2$. Then $\h(A)\cong \solv_2$, 
as (ungraded) Lie algebras. 
\end{example}

The next lemma (see \cite[\S4]{MPPS} for details) identifies the 
set of flat, $\g$-valued connections on a $\cdga$ $(A,d)$ with the set of 
Lie algebra morphisms from the holonomy Lie algebra of $(A,d)$ 
to $\g$. 

\begin{lemma}
\label{lem:flat hom}
The canonical isomorphism $A^1\otimes \g \cong  \Hom (A_1,\g)$  
restricts to isomorphisms 
$\F(A,\g) \cong  \Hom_{\Lie} (\h(A), \g)$ and 
$\F^1(A,\g) \cong  \Hom_{\Lie}^1 (\h(A), \g)$. 
\end{lemma}

Here, $\Hom_{\Lie}^1 (\h(A), \g)$ denotes the subset of Lie algebra 
morphisms with at most $1$-dimensional image. 

\section{Resonance varieties}
\label{sect:res}

In this section, we recall the definition of the Aomoto complexes 
associated to a $\cdga$ $(A,d)$ and a representation of a Lie algebra $\g$, 
as well as the resonance varieties associated to these data, following 
the approach from \cite{DPS, DP, MPPS}. 

\subsection{Twisted differentials}
\label{subsec:aomoto}

Let $\theta \colon \g \to \gl (V)$ be a representation 
of our Lie algebra $\g$ in a finite-dimensional, non-zero $\C$-vector 
space $V$.  For each flat connection $\omega\in \F(A,\g)$, 
we make $A\otimes V$ into a cochain complex, 
\begin{equation}
\label{eq:aomoto}
\xymatrixcolsep{22pt}
\xymatrix{(A\otimes V , d_{\omega})\colon  \ 
A^0 \otimes V \ar^(.65){d_{\omega}}[r] & A^1\otimes V 
\ar^(.5){d_{\omega}}[r] 
& A^2\otimes V   \ar^(.55){d_{\omega}}[r]& \cdots },
\end{equation}
using as differential the covariant derivative 
\begin{equation}
\label{eq:dw}
d_{\omega}=d\otimes \id_V + \ad_{\omega},
\end{equation}
where $\ad_{\omega}$ is defined via the Lie semi-direct product 
$V \rtimes_{\theta} \g$.  The flatness condition insures that 
$d_{\omega}^2=0$.  In coordinates, if 
$\omega= \sum_j \eta_j  \otimes x_j$, then 
\begin{equation}
\label{eq=dcoord}
d_{\omega} (\alpha \otimes v)= d \alpha \otimes v + 
\sum_j \eta_j \alpha \otimes \theta (x_j) (v) ,
\end{equation}
for all $\alpha \in A$ and $v\in V$.  

It is readily seen that the multiplication map 
\begin{equation}
\label{eq:mu}
\mu\colon (A,d) \otimes (A\otimes V,d_{\omega}) \to 
(A\otimes V,d_{\omega}), \quad a\otimes (b\otimes v)\mapsto 
ab \otimes v
\end{equation}
defines the structure of a differential $A^{\hdot}$-module 
on the Aomoto complex $(A^{\hdot} \otimes V,d_{\omega})$.
In particular, the graded vector space $H^{\hdot}(A\otimes V, d_{\omega})$ 
is, in fact, a graded module over the ring $H^{\hdot}(A)$.

\subsection{Resonance varieties of a \cdga}
\label{subsec:res cdga}
Associated to the above data are the resonance varieties 
\begin{equation}
\label{eq:rra}
\RR^i_m(A, \theta)= \{\omega \in \F(A,\g)  
\mid \dim_{\C} H^i(A \otimes V, d_{\omega}) \ge  m\}.
\end{equation}
If  $\g$ is finite-dimensional, 
the sets $\RR^i_m(A, \theta)$ are Zariski closed subsets of 
$\F(A,\g)$, for all $i\le q$ and $m\ge 0$.  
In the case when $\g=\C$ and $\theta=\id_\C$, we will simply 
write $\RR^i_m(A)$ for these varieties, viewed as algebraic 
subsets of $H^1(A)$. Clearly,
\begin{equation}
\label{eq:zero}
0\in \RR^i_1(A,\theta) \Leftrightarrow  
0\in \RR^i_1(A)\Leftrightarrow H^i(A)\ne 0.
\end{equation}

When $d=0$, the varieties $\RR^i_m(A)$ 
are homogeneous subsets of $A^1$. This happens  
in the classical case, when $X$ is a path-connected space, 
and $A=H^{\hdot}(X,\C)$ is its cohomology algebra, 
endowed with the zero differential.

In general, though, the resonance varieties of a $\cdga$ 
are not homogeneous sets, even in the rank $1$ case.  

\begin{example}
\label{rem:res stuff}
Let $(A,d)$ be the $\cdga$ from Example \ref{ex:solv}.  
Then $H^1(A)=\C$, while $\RR_1^1(A)=\{0,1\}$.
\end{example}

\begin{lemma}
\label{lem:res f1}
Let $\omega=\eta\otimes g\in \F^1(A, \g)$.  
\begin{enumerate}
\item 
If $\omega \in \RR^i_1(A, \theta)$, then $A^i\neq 0$. 

\item 
Suppose $A^i\neq 0$ and $d=0$.  Then $\omega \in \RR^i_1(A, \theta)$ 
if and only if either $\eta \in \RR^i_1(A)$ or $\det (\theta(g))=0$. 
\end{enumerate}
\end{lemma}

\begin{proof}
The first claim is clear. 
When $d=0$, recall that the rank one resonance variety $\RR^i_1(A)$ is
homogeneous. The second claim then follows from \cite[Corollary 3.6]{MPPS}. 
\end{proof}

\subsection{Resonance varieties of a Lie algebra}
\label{subsec:res lie}

Let $\h$ be a finitely generated Lie algebra, and 
let $\theta\colon \g\to \gl(V)$ 
be a representation of another Lie algebra.  
Associated to these data are the resonance varieties 
\begin{equation}
\label{eq:lie res}
\RR^i_m(\h, \theta)= \{\varphi \in \Hom_{\Lie}(\h,\g)  
\mid \dim_{\C} H^i(\h,  V_{\theta\circ \varphi}) \ge  m\}, 
\end{equation}
where $V_{\theta\circ \varphi}$ denotes the $\C$-vector space 
$V$, viewed as a module over the enveloping algebra $U(\h)$ 
via the representation $\theta\circ \varphi\colon \h \to \gl(V)$.  

Now suppose $\g$ is finite-dimensional. 
Then the resonance varieties $\RR^i_m(\h,\theta)$ 
are Zariski-closed subsets of $\Hom_{\Lie}(\h,\g)$, 
for all $i\le 1$ and $m\ge 0$.

\begin{lemma}[\cite{MPPS}]
\label{lem:lie res}
For each $i\le 1$ and $m\ge 0$, the canonical isomorphism 
$\F(A,\g) \cong \Hom_{\Lie} (\h(A), \g)$  
restricts to an isomorphism 
\begin{equation}
\label{eq:lie res iso}
\RR^i_m(A,\theta) \cong \RR^i_m (\h(A), \theta) . 
\end{equation}
\end{lemma}

\begin{example}
\label{ex:r0}
Let $x_1,\dots, x_n$ be a basis for $A_1$.  Using  
Lemma \ref{lem:lie res} and \cite[Lemma 2.3]{MPPS}, 
we find that 
\begin{equation}
\label{eq:r01}
\RR^0_1(\h(A), \theta)= \Big\{\varphi \in \Hom_{\Lie}(\h(A),\g)  
\:\Big|\,  \bigcap_{i=1}^{n} \ker (\theta \circ \varphi (x_i)) \ne 0 \Big. \Big\}. 
\end{equation}
\end{example}

\section{Products}
\label{sect:prod}

In this section, we study the way the various constructions outlined 
so far behave under (finite) product operations.

\subsection{Holonomy Lie algebra and products} 
\label{subsec:holo prod}

Let $(A,d)$ and $(\bar{A},\bar{d})$ be two \cdga's.  
The tensor product of these two $\C$-vector spaces, 
$A\otimes \bar{A}$, is again a \cdga, with grading 
$(A\otimes \bar{A})^{q}=\bigoplus_{i+j=q} A^i \otimes \bar{A}^j$, 
multiplication $(a\otimes \bar{a})\cdot (b\otimes \bar{b})=
(-1)^{\abs{\bar{a}}\abs{b}} (ab\otimes \bar{a}\bar{b})$, and differential 
$D$ given on homogeneous elements by $D(a\otimes \bar{a})=
da \otimes \bar{a} +(-1)^{\abs{a}} a\otimes \bar{d}\bar{a}$. 

The definition is motivated by the cartesian product of 
spaces, in which case the K\"{u}nneth formula gives an isomorphism
\begin{equation}
\label{eq:kunneth}
(H^{\hdot}(X \times \overline{X},\C), D=0) \cong (H^{\hdot}(X,\C), d=0) 
\otimes (H^{\hdot}(\overline{X},\C), \overline{d}=0). 
\end{equation}

In \cite[\S9]{DPS}, we gave a product formula for holonomy Lie  
algebras in the $1$-formal case.  We now extend this formula 
to \cdga's with non-zero differential.  

\begin{proposition}
\label{prop:prodpres}
Let $A$ and $\bar{A}$ be two connected \cdga's.  Then the Lie algebra 
$\h(A\otimes \bar{A})$ is generated by $A_1 \oplus \bar{A}_1$, subject 
to the relations $\partial_A (A_2)=0$,  $\partial_{\bar{A}} (\bar{A}_2)=0$,
and $[A_1 ,\bar{A}_1]=0$. 
\end{proposition}

\begin{proof}
By construction, $(A\otimes \bar{A})^1= A^1 \oplus \bar{A}^1$ and
$(A\otimes \bar{A})^2= A^2 \oplus \bar{A}^2 \oplus (A^1 \otimes \bar{A}^1)$.
Plainly, $D^1$ restricts to $d^1$ on $A^1$ and to $\bar{d}^1$ on $\bar{A}^1$.
It is readily seen that the multiplication map on $A\otimes \bar{A}$ restricts to 
the multiplication maps on $A^1 \wedge A^1$ on $\bar{A}^1 \wedge \bar{A}^1$, 
respectively, and to the identity map on $A^1 \otimes \bar{A}^1$. By taking 
duals, we conclude that $\h(A\otimes \bar{A})$ has the 
asserted presentation. 
\end{proof}

\begin{corollary}
\label{cor:holoprod}
The holonomy Lie algebra of a tensor product of \cdga's  
is isomorphic to the (categorical) product of the respective 
holonomy Lie algebras,
\[
\h(A\otimes \bar{A}) \cong \h(A) \times \h(\bar{A}).
\]
\end{corollary}

\subsection{Flat connections and products} 
\label{subset:flat prod}

Proposition \ref{prop:prodpres} also yields a formula for the representation 
variety of a tensor product of \cdga's.  

\begin{corollary}
\label{cor:homprod1}
For any Lie algebra $\g$, 
\begin{eqnarray*}
\Hom_{\Lie}(\h(A\otimes \bar{A}),\g) = \{ (\varphi, \bar{\varphi}) \in
\Hom_{\Lie}(\h(A),\g) \times \hspace*{0.8in}
\\[3pt]
 \Hom_{\Lie}(\h(\bar{A}),\g) \mid  
 [\varphi (x), \bar{\varphi}(\bar{x})]= 0 , \; \forall (x, \bar{x}) \in A_1 \oplus \bar{A}_1 \}.
\end{eqnarray*}
Furthermore, if $\g$ is abelian, then
\[
\Hom_{\Lie}(\h(A\otimes \bar{A}),\g) =
\Hom_{\Lie}(\h(A),\g) \times \Hom_{\Lie}(\h(\bar{A}),\g).
\]
\end{corollary}

For the simple Lie algebra $\sl_2$ and its Borel subalgebra $\solv_2$, 
the above corollary can be made more explicit.

\begin{corollary}
\label{cor:homprod2}
If $\g=\sl_2$ or $\solv_2$, then
\begin{eqnarray*}
\Hom_{\Lie}(\h(A\otimes \bar{A}),\g) =
\{0\}\times  \Hom_{\Lie}(\h(\bar{A}),\g) \cup \hspace*{0.5in}\\
\Hom_{\Lie}(\h(A),\g) \times \{0\} \cup
\Hom_{\Lie}^1(\h(A\otimes \bar{A}),\g).
\end{eqnarray*}
\end{corollary}

\begin{proof}
The inclusion $\supseteq$ is clear. To prove the reverse inclusion, 
fix bases $\{ x_1,\dots,x_n\}$ and $\{ \bar{x}_1,\dots, \bar{x}_m\}$ 
for $A_1$ and $\bar{A}_1$.  
Let $\varphi \colon \h(A\otimes \bar{A}) \to \g$ be a morphism 
of Lie algebras, and suppose there are indices $i$ and $j$ 
such that $\varphi (x_i)\neq 0$ and $\varphi (\bar{x}_j)\neq 0$.  
We need to prove that the family 
$\{ \varphi (x_1), \dots$,  $\varphi (x_n), \varphi (\bar{x}_1), 
\dots, \varphi (\bar{x}_m) \}$ has rank $1$.  

We know from Corollary \ref{cor:homprod1} that 
$[\varphi (x_k), \varphi (\bar{x}_l)]=0$, for all $k\in [n]$ 
and $l\in [m]$. Now note that, for any $0\neq g,h \in \g$, 
the following holds: $[g,h]=0$ if and only if $g=\lambda h$, 
for some $\lambda \in \C^{\times}$.  The desired conclusion 
is now immediate.
\end{proof}

\subsection{Resonance and products} 
\label{subsec:resprod}

We now turn to the jump loci of a tensor product of 
\cdga's.  We start with a general upper bound for the depth $1$ 
resonance varieties.

\begin{theorem}
\label{thm:res prod}
For any representation $\theta\colon \g\to \gl(V)$,
\[
\RR^q_1(A\otimes \bar{A}, \theta) \subseteq 
\bigg( \bigg( \bigcup_{i\le q} \RR^i_1(A,\theta)\bigg) \times 
\bigg( \bigcup_{j\le q} \RR^j_1(\bar{A},\theta)\bigg) \bigg)  \cap
\F(A\otimes \bar{A}, \g). 
\]
\end{theorem}

\begin{proof}
By Lemma \ref{lem:flat hom} and Corollary \ref{cor:homprod1}, 
every element $\Omega \in \F(A\otimes \bar{A}, \g)$ can be 
written as $\Omega=\omega + \bar{\omega}$,
for some $\omega \in \F(A, \g)$ and $\bar{\omega} \in \F(\bar{A}, \g)$.
Setting up a first-quadrant double complex with 
$E^{i,j}_0=A^i\otimes \bar{A}^j \otimes V$, 
horizontal differential $d_{\omega}\colon E^{i,j}_0 \to E^{i+1,j}_0$, 
and vertical differential $d_{\bar\omega}\colon E^{i,j}_0 \to E^{i,j+1}_0$, 
we obtain  spectral sequences starting at 
\begin{equation}
{}_{\rm hor}E^{i,j}_1=H^i(A\otimes V, d_{\omega})\otimes \bar{A}^j
\:\ {\rm and} \:\
{}_{\rm vert}E^{i,j}_1=A^i\otimes H^j(\bar{A}\otimes V, d_{\bar\omega}),
\end{equation}
respectively, and converging to 
$H^{i+j}( A\otimes\bar{A}\otimes V, d_{\Omega})$.
See (\ref{eq=dcoord}). 

Consequently, if either $H^{\le q} ( A\otimes V, d_{\omega})$ or 
$H^{\le q} ( \bar{A}\otimes V,  d_{\bar\omega})$ vanishes, then 
$H^{q} ( A\otimes\bar{A}\otimes V, d_{\Omega})=0$.  
In view of definition (\ref{eq:rra}), this completes the proof.
\end{proof}

In general, the inclusion from Theorem \ref{thm:res prod} is strict.  
We illustrate this phenomenon with a simple example.

\begin{example}
\label{ex:strict}
Let $A$ be the exterior algebra on a single 
generator in degree $1$, let $\g=\gl_2$, and let $\theta=\id_\g$. 
Using Example \ref{ex:r0} and Corollary \ref{cor:homprod1}, 
we see that $\RR^0_1(A,\theta)=\{ g\in \gl_2 \mid \det(g) = 0\}$, yet 
\[
\RR^0_1(A\otimes A,\theta)=\{ (g,h)\in \gl_2 \times \gl_2 \mid [g,h]=0, \,
\rank (g \,|\, h) <2\},
\] 
which is a proper subset of 
$(\RR^0_1(A,\theta) \times \RR^0_1(A,\theta)) \cap \F(A\otimes A, \g)$. 
\end{example}

\subsection{Product formulas for resonance} 
\label{subsec:prod res formula}

Under certain additional hypotheses, the upper bound from 
Theorem \ref{thm:res prod} may be improved to an equality.
First, as shown in \cite[Proposition 13.1]{PS-plms}, such an 
equality holds in the formal, rank $1$ case.

\begin{proposition}[\cite{PS-plms}]
\label{prop:prod1}
Assume both $A$ and $\bar{A}$ have zero differential. Then
\[
\RR^q_1(A\otimes \bar{A}) = 
\bigcup_{i+j=q} \RR^i_1(A) \times \RR^j_1(\bar{A}).
\]
\end{proposition} 

Using this result, we now show that an analogous resonance 
formula holds for the non-abelian Lie algebras $\sl_2$ and $\solv_2$.

\begin{theorem}
\label{thm:prodres2}
Assume both $A$ and $\bar{A}$ have zero differential, and 
$\g=\sl_2$ or $\solv_2$. Then, for any representation $\theta\colon \g\to \gl(V)$, 
\[
\RR^q_1(A\otimes \bar{A}, \theta) = 
\Big( \bigcup_{i+j=q} \RR^i_1(A, \theta) \times \RR^j_1(\bar{A}, \theta) \Big)
\cap \F(A\otimes \bar{A}, \g). 
\]
\end{theorem} 

\begin{proof}
Proof of inclusion $\subseteq$. 
Let $\Omega \in \RR^q_1(A\otimes \bar{A}, \theta)$. 
In view of Lemma \ref{lem:res f1} and Corollary \ref{cor:homprod2}, 
there are two cases to consider:  
either $\Omega = \omega \in \F(A, \g)$ (the case 
$\Omega = \bar{\omega} \in \F(\bar{A}, \g)$ being similar), or
$\Omega =(\eta + \bar{\eta}) \otimes g$. 

In the first case,  $H^{q}( A\otimes\bar{A}\otimes V, d_{\Omega})= 
\bigoplus_{i+j=q} H^i(A\otimes V, d_{\omega})\otimes \bar{A}^j$.
Hence, $\Omega \in \RR^q_1(A\otimes \bar{A}, \theta)$ if and 
only if there exist indices $i$ and $j$ 
with $i+j=q$ such that $\omega \in \RR^i_1(A, \theta)$ 
and $\bar{A}^j\neq 0$. Therefore, $\Omega = \omega +0 \in 
\RR^i_1(A, \theta) \times \RR^j_1(\bar{A}, \theta)$, as asserted;
see (\ref{eq:zero}). 

In the second case, suppose $\Omega \in \RR^q_1(A\otimes \bar{A}, \theta)$. 
There exist then indices $i$ and $j$ with $i+j=q$ such that $A^i$ and  
$\bar{A}^j$ are non-zero. If $\det \theta(g)=0$, then we must have 
$\eta \otimes g\in \RR^i_1(A, \theta)$ and $\bar{\eta} \otimes g\in 
\RR^j_1(\bar{A}, \theta)$, and so we are done. Otherwise, 
Proposition \ref{prop:prod1} implies that 
$\eta + \bar{\eta} \in \bigcup_{i+j=q} \RR^i_1(A) \times \RR^j_1(\bar{A})$. 
Therefore, $\eta \otimes g\in \RR^i_1(A, \theta)$ and 
$\bar{\eta} \otimes g\in \RR^j_1(\bar{A}, \theta)$, 
for some  $i$ and $j$ with $i+j=q$. This 
completes the first half of the proof.

Proof of inclusion $\supseteq$. Again, we have to analyze the 
two cases from Corollary \ref{cor:homprod2}.
When $\Omega = \omega +0$ with $\omega  \in \F(A, \g)$, 
we know that $\omega \in \RR^i_1(A, \theta)$
and  $\bar{A}^j\neq 0$, for some  $i$ and $j$ with $i+j=q$. 
As before, we infer that 
$H^{q}( A\otimes\bar{A}\otimes V, d_{\Omega}) \supseteq  
H^i(A\otimes V, d_{\omega})\otimes \bar{A}^j \neq 0$.  
Hence, $\Omega \in \RR^q_1(A\otimes \bar{A}, \theta)$, as claimed. 

Finally, assume that $\Omega =(\eta + \bar{\eta}) \otimes g$, 
where $\eta \otimes g\in \RR^i_1(A, \theta)$ and 
$\bar{\eta} \otimes g\in \RR^j_1(\bar{A}, \theta)$,
for some  $i$ and $j$ with $i+j=q$. When $\det (\theta(g)) \neq 0$, 
we deduce that $\eta \in \RR^i_1(A)$ and $\bar{\eta} \in \RR^j_1(\bar{A})$, 
hence $\eta + \bar{\eta} \in \RR^q_1(A\otimes \bar{A})$, by
Proposition \ref{prop:prod1}. Consequently,  
$\Omega \in \RR^q_1(A\otimes \bar{A}, \theta)$,
as asserted. If $\det (\theta(g))=0$, the fact that $A^i$ and $\bar{A}^j$ 
are both non-zero forces $(A\otimes \bar{A})^q \neq 0$.  
Hence, once again, $\Omega \in \RR^q_1(A\otimes \bar{A}, \theta)$, 
and we are done.
\end{proof}

\section{Coproducts}
\label{sect:coprod}

In this final section, we study the way our various constructions 
behave under (finite) coproducts.

\subsection{Holonomy Lie algebras and coproducts}
\label{subsec:holo coprod}

Let $A=(A^{\hdot},d)$ and $\bar{A}=(\bar{A}^{\hdot},\bar{d})$ be 
two connected \cdga's.  Their wedge sum, $A\vee \bar{A}$, is a 
new connected \cdga, whose underlying graded vector space in 
positive degrees is $A^+\oplus \bar{A}^+$, with multiplication 
$(a,\bar{a})\cdot (b,\bar{b}) = (ab, \bar{a}\bar{b})$, and differential 
$D=d+\bar{d}$.

The definition is motivated by the wedge operation on pointed 
spaces, in which case we have a well-known isomorphism
\begin{equation}
\label{eq:coho wedge}
(H^{\hdot}(X \vee \overline{X}), D=0)\cong (H^{\hdot}(X), d=0) \vee 
(H^{\hdot}(\overline{X}), \overline{d}=0). 
\end{equation}

We now extend the coproduct formula for $1$-formal spaces 
from \cite[\S9]{DPS}, as follows. 

\begin{proposition}
\label{prop:coprodpres}
The holonomy Lie algebra $\h(A\vee \bar{A})$ is generated by 
$A_1 \oplus \bar{A}_1$, with relations $\partial_A (A_2)=0$ and 
$\partial_{\bar{A}} (\bar{A}_2)=0$.
\end{proposition}

\begin{proof}
By construction, $(A\vee \bar{A})^1= A^1 \oplus \bar{A}^1$,
$(A\vee \bar{A})^2= A^2 \oplus \bar{A}^2$, and $D^1= d^1 \oplus \bar{d}^1$. 
Moreover, the multiplication map on $A\vee \bar{A}$ restricts to 
the multiplication maps on $A^1 \wedge A^1$ and $\bar{A}^1 \wedge \bar{A}^1$, 
respectively, and is zero when restricted to $A^1 \otimes \bar{A}^1$. 
The conclusion follows at once.
\end{proof}

\begin{corollary}
\label{cor:holo coprod}
The holonomy Lie algebra of a wedge sum of \cdga's  
is isomorphic to the (categorical) coproduct of the respective 
holonomy Lie algebras,
\[
\h(A\vee \bar{A}) \cong \h(A) \coprod \h(\bar{A}).
\]
\end{corollary}

\subsection{Resonance and coproducts}
\label{subset:flat coprod}

As shown in \cite[Proposition 13.3]{PS-plms}, the classical 
resonance varieties behave nicely with respect to wedges of spaces. 
Let us recall this result, in a form adapted to our purposes. 

\begin{proposition}[\cite{PS-plms}]
\label{cor:coprod}
Assume both $A$ and $\bar{A}$ have zero differential.   Then, 
for all $i>1$,  
\[
\RR^i_1(A\vee \bar{A}) = \RR^i_1(A)\times H^1(\bar{A})   
\cup 
H^1(A)  \times \RR^i_1(\bar{A}).
\]
If, moreover, $b_1(A)> 0$ and $b_1(\bar{A})> 0$, then 
\[
\RR^1_1(A\vee \bar{A})  = H^1(A) \times H^1(\bar{A}).
\]
\end{proposition}

Our goal for the rest of this section will be to extend the above proposition 
to the non-abelian setting, for \cdga's with non-zero differential.  
To that end, let $\g$ be a Lie algebra, and let 
$\omega\in A^1\otimes \g$ and $\bar\omega\in \bar{A}^1\otimes \g$.  
Set $\Omega=\omega+\bar\omega \in  (A\vee \bar{A})^1 \otimes \g$. 

\begin{lemma}
\label{lem:flat coprod}
$\Omega$ is a flat connection if and only if both 
$\omega$ and $\bar\omega$ are flat. 
\end{lemma}

\begin{proof}
By definition of multiplication in $A\vee \bar{A}$, we have that 
$a\cdot \bar{a}=0$ for every $a\in A^+$ and $\bar{a}\in\bar{A}^+$.  
Hence, $[\omega,\bar{\omega}]=0$, and the conclusion follows.
\end{proof}

Now let $\theta\colon \g\to \gl(V)$ be a representation. 
Given an element $\omega\in \F(A,\g)$, we write 
$Z^i_{\omega}=\ker(d_{\omega}\colon A^i\otimes V \to A^{i+1}\otimes V)$ 
and $B^i_{\omega}=\im(d_{\omega} \colon A^{i-1}\otimes V \to A^i\otimes V)$, 
and set $H^i_{\omega}=Z^i_{\omega}/B^i_{\omega}$. 

\begin{lemma}
\label{lem=dwedge}
For $i>0$,
\[
d^i_{\Omega}= d^i_{\omega} \oplus d^i_{\bar{\omega}} \colon 
(A^i \otimes V) \oplus (\bar{A}^i \otimes V) \longrightarrow
(A^{i+1} \otimes V) \oplus (\bar{A}^{i+1} \otimes V) ,
\]
while for $i=0$
\[
d^0_{\Omega}= (d^0_{\omega}, d^0_{\bar{\omega}}) \colon 
(A\vee \bar{A})^0 \otimes V \cong  V \longrightarrow
(A^1 \otimes V) \oplus (\bar{A}^1 \otimes V).
\]
\end{lemma}

\begin{proof}
Both claims follow from (\ref{eq=dcoord}) and the construction of $A\vee \bar{A}$,
by straightforward direct computation.
\end{proof}

\begin{corollary}
\label{cor:coprodg1}
For each $i>1$ and for any representation $\theta\colon \g\to \gl(V)$, 
\[
\RR^i_1(A\vee \bar{A}, \theta) = 
\RR^i_1(A, \theta) \times \F(\bar{A}, \g) \cup
\F(A, \g) \times \RR^i_1(\bar{A}, \theta).
\]
\end{corollary}

\begin{proof}
By Lemma \ref{lem=dwedge}, $H^i_{\Omega} \cong 
H^i_{\omega} \oplus H^i_{\bar{\omega}}$.  Using this isomorphism, 
the desired conclusion follows from Lemma \ref{lem:flat coprod}. 
\end{proof}

\subsection{A coproduct formula for degree $1$ resonance}
\label{subsec:res coprod}

To conclude, we compute the degree $1$ resonance variety 
of a wedge sum, $\RR^1_1(A\vee \bar{A}, \theta)$.  We start 
with two lemmas.

\begin{lemma}
\label{lem:phi map}
There is a surjective homomorphism 
\[
\xymatrix{
H^1((A\vee \bar{A})\otimes V, d_{\Omega}) \ar@{->>}^(.41){\Phi}[r]
&H^1(A\otimes V, d_{\omega}) \oplus 
H^1(\bar{A}\otimes V, d_{\bar{\omega}})}, 
\]
whose kernel is isomorphic to $(B^1_{\omega} \oplus B^1_{\bar{\omega}})/ 
\im ((d^0_{\omega}\oplus d^0_{\bar{\omega}})\circ \Delta)$, 
where $\Delta \colon V \to V \oplus V$ is the diagonal map.
\end{lemma}
\begin{proof}
Follows from Lemma \ref{lem=dwedge}.
\end{proof}

\begin{lemma}
\label{lem:ker0}
The homomorphism $\Phi$ is injective if and only if 
$V=Z^0_{\omega}+Z^0_{\bar\omega}$.
\end{lemma}

\begin{proof}
Start by noting that $V \oplus V= \im (\Delta) \oplus (V\times 0)$. 
A standard linear algebra argument, then, finishes the proof. 
\end{proof}
 
\begin{theorem}
\label{thm:res coprod}
Suppose both $b_1(A)$ and $b_1(\bar{A})$ are positive, 
and at least one of them is greater than $1$. Then, 
for any representation $\theta\colon \g\to \gl(V)$, 
\[
\RR^1_1(A\vee \bar{A}, \theta) = 
\F(A\vee \bar{A}, \g).
\]
\end{theorem}

\begin{proof}
Set $r=\dim V$. Using our hypothesis,
we may assume that $b_1(A)>1$ and $b_1(\bar{A})\ge 1$. 
Supposing $H^1_{\Omega}=0$ for
some $\Omega =\omega + \bar{\omega}\in \F (A\vee \bar{A}, \g)$, 
we derive a contradiction, as follows. 

Lemma \ref{lem:phi map} implies that 
$Z^1_{\omega}=B^1_{\omega}$ 
and $Z^1_{\bar{\omega}}=B^1_{\bar{\omega}}$. 
Furthermore, the discussion from \S\ref{subsec:aomoto}  
shows that $Z^1(A)\otimes Z^0_{\omega} \subseteq Z^1_{\omega}$ and 
$Z^1(\bar{A})\otimes Z^0_{\bar{\omega}} \subseteq Z^1_{\bar{\omega}}$.
Hence, 
\[
r- \dim Z^0_{\omega}= \dim B^1_{\omega}=  
\dim Z^1_{\omega}\ge b_1(A) \cdot \dim Z^0_{\omega}, 
\]
and so $\dim Z^0_{\omega}\le r/(b_1(A)+1)< r/2$. Similarly,
$\dim Z^0_{\bar{\omega}}\le r/2$.

Using again Lemma \ref{lem:phi map}, we deduce that $\Phi$ 
must be injective.  By Lemma \ref{lem:ker0}, 
\[
r= \dim (Z^0_{\omega}+ Z^0_{\bar{\omega}}) \le 
\dim Z^0_{\omega}+ \dim Z^0_{\bar{\omega}}<r,
\]
a contradiction.
\end{proof}

\begin{acknowledgement}
This work was started while the two authors visited 
the Max Planck Institute for Mathematics in Bonn in April--May 2012. 
The work was pursued while the second author visited the Institute of 
Mathematics of the Romanian Academy in June, 2012 and 
June, 2013, and MPIM Bonn in September--October 2013.  
Thanks are due to both institutions for their hospitality, 
support, and excellent research atmosphere. 
\end{acknowledgement}

\end{document}